\documentclass[12pt]{amsart}
\usepackage{latexsym}
\usepackage{amsthm}
\usepackage{amsmath}
\usepackage{amssymb}
\usepackage{verbatim}
\usepackage{color}
\newtheorem{theorem}{Theorem}%[section]

\newtheorem{corollary}{Corollary} 
\newtheorem{lemma}{Lemma}
\def \n{\noindent }
\def \bs{\bigskip}

\def \Z{\mathbb Z}

\def \Z{\mathbb Z}

\title{A zero-sum theorem over $\mathbb{Z}$} 
\author{M. Sahs, P. Sissokho, and J. Torf}
\address{4520 Mathematics Department, Illinois State University, 
Normal, Illinois 61790--4520, U.S.A.}
\email{\{mlsahs|psissok|jntorf\}@ilstu.edu}
\begin{document}
\begin{abstract}
A {\em zero-sum} sequence of integers is a sequence of nonzero terms that sum to $0$.
Let $k>0$ be an integer and let $[-k,k]$ denote the set of all nonzero 
integers between $-k$ and $k$. Let $\ell(k)$ be the smallest integer $\ell$ such 
that any zero-sum sequence with  elements from $[-k,k]$ and length greater than $\ell$ 
contains a proper nonempty zero-sum subsequence. In this paper, we prove a more 
general result which implies that $\ell(k)=2k-1$ for $k>1$.
\end{abstract}
\maketitle
%==========================================================================
\section{Introduction}\label{sec:intro}
%===============================================================================
For any multiset $S$, let $|S|$ denote the number of elements in $S$, let $\max(S)$ denote
the maximum element in $S$, and let $\Sigma S=\sum_{s\in S}s$. 
Let $A$ and $B$ be nonempty multisets of positive integers. The pair $\{A,B\}$ is said to be 
{\em irreducible} if $\Sigma A=\Sigma B$, and 
for every nonempty proper mutisubsets $A'\subset A$ and $B'\subset B$, 
$\Sigma A'\not=\Sigma B'$ holds. If $\{A,B\}$ fails to be irreducible, we say that it is 
{\em reducible}. It is easy to see that if $\{A,B\}$ is irreducible, then $A\cap B=\emptyset$
or $|A|=|B|=1$. 

We define the {\em length} of $\{A,B\}$ as
\[\ell(A,B)=|A|+|B|.\]
An irreducible pair $\{A,B\}$ is said to be {\em $k$-irreducible} if $\max(A\cup B)\leq k$. 
We define
\begin{equation}\label{def:lk}
\ell(k)=\max\limits_{\{A,B\}}\ell(A,B),
\end{equation}
where the maximum is taken over all $k$-irreducible pairs $\{A,B\}$. 

For $k>1$, let 
\begin{equation}\label{exp:lb}
A=\{\underbrace{k,\ldots,k}_{k-1}\} \mbox{ and } B=\{\underbrace{k-1,\ldots,k-1}_k\}.
\end{equation}
Then $\{A,B\}$ is $k-$irreducible and $\ell(A,B)=2k-1$. This implies that $\ell(k)\geq 2k-1$. 
El-Zanati, Seelinger, Sissokho, Spence, and Vanden Eynden introduced $k$-irreducible pairs
in connection with their work on irreducible $\lambda$-fold partitions 
(e.g., see~\cite{ESSSV}). They also conjectured that $\ell(k)=2k-1$. 
In the our main theorem below, we prove a more general result which implies this conjecture in our main theorem below.
\begin{theorem}\label{thm:1}
If $\{A,B\}$ is an irreducible pair, then $|A|\leq \max(B)$ and $|B|\leq \max(A)$.
Consequently, $\ell(k)=2k-1$ if $k>1$.
\end{theorem}
One may naturally ask which $k-$irreducible pairs  $\{A,B\}$ achieve the 
maximum possible length. We answer this question in the the following corollary.
\begin{corollary}\label{cor:1}
Let $k>1$ be an integer. A $k-$irreducible pair $\{A,B\}$ has (maximum possible) 
length $\ell(A,B)=2k-1$ if and only if $\{A,B\}$ is the pair shown in~$(2)$.
\end{corollary}
%====================================================
A {\em zero-sum} sequence is a sequence of nonzero terms that sum to $0$.
A zero-sum sequence is said to be {\em irreducible} if it does not contain
a proper nonempty zero-sum subsequence. 
Given a zero-sum sequence $\tau$ with elements from $[-k,k]$, 
let $A_\tau$ be the multiset of all positive integers from $\tau$ and $B_\tau$ be the 
multiset containing the absolute values of all negative integers from $\tau$. Then the sequence 
$\tau$ is irreducible if and only if the pair $\{A_\tau,B_\tau\}$ is irreducible. 

Let $k$ be a positive integer, and let $[-k,k]$ denote the set of all nonzero 
integers between $-k$ and $k$. Then the number $\ell(k)$ defined in~\eqref{def:lk} is also equal to 
the smallest integer $\ell$ such that any zero-sum sequence with  elements from $[-k,k]$ 
and length greater than $\ell$ contains a proper nonempty zero-sum subsequence. Moreover,
it follows from Theorem~\ref{thm:1} that $\ell(k)=2k-1$.

Let $G$ be a finite (additive) abelian group of order $n$. The {\em Davenport constant} of $G$,
denoted by $D(G)$, is the smallest integer $m$ such that any sequence of 
elements from $G$ with length $m$ contains a nonempty zero-sum subsequence.
Another key constant, $E(G)$, is the smallest integer $m$ such that any sequence of 
 elements from $G$ with length $m$ contains a zero-sum subsequence
of length exactly $n$. The constant $E(G)$ was inspired by the well-known result 
of Erd\"os, Ginzburg, and Giv~\cite{EGZ}, which states that $E(\Z/n\Z)=2n-1$. Subsequently, 
Gao~\cite{G} proved that $E(G)=D(G)+n-1$. There is a rich literature of research 
dealing with the constants $D(G)$ and $E(G)$. We refer the interested reader to the survey papers 
of Caro~\cite{C} and Gao--Geroldinger~\cite{GG} for further information.

By rephrasing our main theorem using the language of zero-sum sequence, we can view 
it as a zero-sum theorem.
Whereas zero-sum sequences are traditionally studied for finite abelian groups
such as $\Z/n\Z$, we consider in this paper zero-sum sequences over the infinite group $\Z$. 

The rest of the paper is structured as follows. In Section~\ref{sec:main}, we prove 
our main results (Theorem~\ref{thm:1} and Corollary~\ref{cor:1}), and in Section~\ref{sec:conc}, 
we end with some concluding remarks. 

%==========================================================================
\section{Proofs of Theorem~\ref{thm:1} and Corollary~\ref{cor:1}}\label{sec:main}\
%\section{Proof of the Main Results (Theorem~\ref{thm:1} and Corolary~\ref{thm:1})}\label{sec:main}\
%==========================================================================
Suppose, we are given a $k-$irreducible pair $\{A,B\}$. We may assume that 
 $A=\{x_1\cdot a_1,x_2\cdot a_2,\ldots, x_n\cdot a_n\}$ and 
$B=\{y_1\cdot b_1,y_2\cdot b_2,\ldots, y_m\cdot b_m\}$, where 
the $a_i$'s and $b_j$'s are all positive integers 
such that $1\leq a_i,b_j\leq k$ for $1\leq i\leq n$, $1\leq j\leq m$.
We also assume that the $a_i$'s (resp. $b_j$'s) are pairwise distinct.
Moreover, $x_i>0$ and $y_j>0$ are the multiplicities of $a_i$ and $b_j$ respectively.
We also assume that the $a_i$'s (resp. $b_j$'s) are pairwise distinct.
For any pair $(a_i,b_j)$, let
\begin{enumerate}
\item $C$ be the multiset obtained from $A$ by: $(i)$ removing one copy of $a_i$, 
and $(ii)$ introducing one copy of $a_i-b_j$ if $a_i>b_j$.
\item $D$ be the multiset obtained from $B$ by: $(i)$ removing one copy of $b_j$, 
and $(ii)$ introducing one copy of $b_j-a_i$ if $b_j>a_i$.
\end{enumerate}
We say that $\{C,D\}$ is {\em $(a_i,b_j)$-derived}  from $\{A,B\}$.
We also call the above process an {\em $(a_i,b_j)$-derivation}.
Consider the integers $p>0$, $q>0$, and $z_{ij}\geq 0$ for $p\leq i\leq q$ and $u\leq j\leq v$.
We say that $\{C,D\}$ is $\prod_{i=p}^q\prod_{j=u}^v (a_i,b_j)^{z_{ij}}$-derived
from $\{A,B\}$ if it is obtain by performing on $\{A,B\}$ an $(a_i,b_j)$-derivation 
$z_{ij}$ times for each  $(i,j)$  pair. (If $z_{ij}=0$, then we simply do not perform 
the corresponding $(a_i,b_j)$-derivation.)

We illustrate this operation with the following example. Let 
$A=\{3\cdot 7,2\cdot 1\}=\{7,7,7,1,1\}$ and $B=\{3\cdot 6,5\}=\{6,6,6,5\}$.
Then $\{A,B\}$ is $7$-irreducible.  A $(7,6)^2(7,5)$-derivation of $(A,B)$ yields 
the pair $\{C,D\}$, where $C=\{2,1,1,1,1\}$ and $D=\{6\}$. Note that $\{C,D\}$ 
is $6$-irreducible (thus, $7$-irreducible).

In general, the order in which the derivation is done makes a difference.
For example, if $A=\{5,5\}$ and $B=\{2,2,2,2,2\}$, then we can do a $(5,2)$ derivation 
followed by a $(3,2)$-derivation on $\{A,B\}$, but not in reverse order.
However, all the derivation used in our proofs can be done in any order.

We will use the following lemma.
\begin{lemma}\label{lem:1}\
Let $A=\{x_1\cdot a_1,x_2\cdot a_2,\ldots, x_n\cdot a_n\}$ and 
$B=\{y_1\cdot b_1,y_2\cdot b_2,\ldots, y_m\cdot b_m\}$ be multisets, where 
the $a_i$'s and $b_i$'s are all positive integers 
such that $1\leq a_i,b_j\leq k$ for $1\leq i\leq n$, $1\leq j\leq m$.
Moreover, $x_i>0$ and $y_j>0$ are the multiplicities of $a_i$ and $b_j$ respectively.
Suppose that $\{A,B\}$ is a $k$-irreducible pair with length $|A|+|B|>2$.

\n $(i)$  If $\{C,D\}$ is $(a_i,b_j)$-derived from $(A,B)$, then it is $k-$irreducible.

\n $(ii)$ Let $p>0$, $q>0$, and $z_{ij}\geq 0$ for
$p\leq i\leq q$ and $u\leq j\leq v$, be integers.
Assume that $\sum_{j=u}^v z_{ij}\leq x_i$ and  $\sum_{i=p}^q z_{ij}\leq y_j$.
If $\{C,D\}$ is $\prod_{i=p}^q\prod_{j=u}^v (a_i,b_j)^{z_{ij}}$-derived 
from $\{A,B\}$, then it is $k-$irreducible.
\end{lemma}
\begin{proof}
We first prove $(i)$. Without loss of generality, we may assume 
that $a_i>b_j$ since the proof is similar for $a_i<b_j$. Then
\begin{equation}\label{eq:lem.1}
C=\left(A-\{a_i\}\right)\cup \{a_i-b_j\} \mbox{ and } D=B-\{b_j\},
\end{equation}
are nonempty since $|A|+|B|>2$. Since $\{A,B\}$ is irreducible, we have
\[ \Sigma A=\Sigma B\Rightarrow  \Sigma C=\Sigma A-a_i+(a_i-b_j)=\Sigma B -b_j=\Sigma D.\]
Assume that $\{C,D\}$ is reducible. Then, there exist nonempty proper   subsets
 $C'\subset C$ and $D'\subset D$ such that $\Sigma C'=\Sigma D'$. 
Let $\overline{C}'=C-C'$ and $\overline{D}'=D-D'$.
Then $\overline{C}'\subset C$ and $\overline{D}'\subset D$ are also nonempty proper 
subsets that satisfy $\Sigma \overline{C}'=\Sigma \overline{D}'$. However, it follows 
from the definition of $C$ in~\eqref{eq:lem.1} that either $C'$ or $\overline{C}'$ is a proper 
subset of $A$, since $a_i-b_j$ cannot be in both $C'$ and $\overline{C}'$. It also follows from 
the definition of $D$ in~\eqref{eq:lem.1} that 
both $D'$  and $\overline{D}'$ are proper subsets of $B$. Thus, either the 
subset pair $\{C',D'\}$ or $\{\overline{C}',\overline{D}'\}$ is a witness to the reducibility 
of $\{A,B\}$. This contradicts the fact that $\{A,B\}$ is irreducible. Hence, if $\{A,B\}$ is 
irreducible, then $\{C,D\}$ is also irreducible. In addition, it follows from~\eqref{eq:lem.1} 
that $\max(C)\leq \max(A)$ and $\max(D)\leq \max(B)$. Hence, if $\{A,B\}$ is $k$-irreducible, then
$\{C,D\}$ is also $k$-irreducible.

To prove $(ii)$, observe that we can apply $(i)$ recursively by performing (in any order) on $\{A,B\}$ 
an $(a_i,b_j)$-derivation $z_{ij}$ times for each $(i,j)$ pair. The conditions 
on the $z_{ij}$'s guarantee that there are enough pairs $(a_i,b_j)$ in $A\times B$ to independently
perform all the $(a_i,b_j)$-derivations for $p\leq i\leq q$ and $u\leq j\leq v$.
\end{proof}

We will also need the following basic lemma.
\begin{lemma}\label{lem:2}
Let  $x_i$ and $y_j$ be positive integers, where $1\leq i\leq n$  and  $1\leq j\leq m$.
If $t<m$ be a positive integer such that 
\[\sum_{j=1}^t y_j\leq \sum_{i=1}^n x_i \mbox{ and } \sum_{j=1}^{t+1} y_j>\sum_{i=1}^n x_i,\]
then there exist integers $z_{ij}\geq0 $, $1\leq i\leq n$ and $1\leq j\leq t+1$, 
such that 
\[\sum_{i=1}^n z_{ij}=y_j\mbox{ for $1\leq j\leq t$, }\;
z_{i,t+1}=x_i-\sum_{j=1}^t z_{ij}\geq0, \mbox{ and } y_{t+1}>\sum_{i=1}^n z_{i,t+1}.\]
\end{lemma}
\begin{proof}
For each $j$, $1\leq j\leq t+1$,  consider $y_j$ marbles of color $j$.
For each $i$, $1\leq j\leq n$,  consider a bin with capacity $x_i$ (i.e., 
it can hold $x_i$ marbles).
Since $p=\sum_{j=1}^t y_j\leq \sum_{i=1}^n x_i=q$,  we can 
distribute all the $p$ marbles  into the $n$ bins (with total capacity $q$) without 
 exceeding the capacity of any given bin.
 Since $p+y_{t+1}=\sum_{j=1}^{t+1} y_j> q$,  we can use the additional $y_{t+1}$ 
marbles to top off  the bins that were not already full.

Now define $z_{ij}$  to be the number of marbles in bin $i$ that have color $j$.
Then the  $z_{ij}$'s  satisfy the required properties.
\end{proof}

\bs
We now prove our main theorem.
\begin{proof}[Proof of Theorem~\ref{thm:1}]\

Let $\{A,B\}$ be a $k-$irreducible pair. We can write
$A=\{x_1\cdot a_1,x_2\cdot a_2,\ldots, x_n\cdot a_n\}$ and 
$B=\{y_1\cdot b_1,y_2\cdot b_2,\ldots, y_m\cdot b_m\}$, where 
the $a_i$'s and $b_i$'s are all positive integers 
such that $1\leq a_i,b_j\leq k$ for $1\leq i\leq n$, $1\leq j\leq m$.
Moreover, $x_i>0$ and $y_j>0$ are the multiplicities of $a_i$ and $b_j$ respectively.
Consequently, we may assume that the $a_i$'s (resp. $b_j$'s) are pairwise distinct.
Without loss of generality, we may also assume that 
\begin{equation}\label{eq:gen-ass}
a_1>\ldots>a_n \mbox{ and } b_1>\ldots>b_m.
\end{equation}

We shall prove by induction on $r=\max(A)+\max(B)\geq 2$ that 
\begin{equation}\label{eq:ind}
|A|\leq \max(B)\quad \mbox{ and }\quad  |B|\leq \max(A).
\end{equation}
If $r=2$, then $k\leq 2$ and the only possible irreducible pair is $\{\{1\},\{1\}\}$.
Thus, the inductive statement~\eqref{eq:ind} is clearly true.

If $a_i=b_j$ for some pair $(i,j)$, then $A=\{a_i\}=B$. (Otherwise, $A'=\{a_i\}\subset A$  and 
$B'= \{a_i\}\subset B$  are  nonempty proper subsets satisfying $\Sigma A'=\Sigma B'$, 
which contradicts the irreducibility of $\{A,B\}$.) Moreover, $|A|=|B|=1\leq a_i=\max(A)=\max(B)$
holds. Since $k>1$, we further obtain $\ell(A,B)=|A|+|B|=2<2k-1$.

So we can assume that $A\cap B=\emptyset$. Without loss of generality, we may also 
assume that $\max(A)=a_1>b_1=\max(B)$.

\bs
Suppose that the theorem holds for all $k-$irreducible pairs $\{C,D\}$ with 
$2\leq r'=\max(  C)+\max( D)<r$. To prove the inductive step, we consider two parts. 

\bs\n{\bf Part I:}
In this part, we show  $|A|\leq \max(B)$ . We consider two cases.

\bs\n {\bf Case 1:}  $y_1>x_1$.

Since $y_1>x_1$, we can perform an $(a_1,b_1)^{x_1}$-derivation from $\{A,B\}$ to 
obtain (by Lemma~\ref{lem:1}) the $k$-irreducible pair $\{C,D\}$, where
\[C=\{x_1\cdot (a_1-b_1),x_2\cdot a_2,\ldots,x_n\cdot a_n\} \]
and $D=\{ (y_1-x_1)\cdot b_1,y_2\cdot b_2,\ldots,y_m\cdot b_m\}$.

Since $r'=\max(C)+\max(D)=\max\{a_1-b_1,a_2\} +b_1<r$, it follows from 
the induction hypothesis that 
\begin{equation}\label{eq:c1}
|C|=\sum_{i=1}^n x_i\leq \max(D)=b_1.
\end{equation}
It follows from~\eqref{eq:c1} that $|A|=\sum_{i=1}^n x_i=|C|\leq b_1$ as required.

\bs\n {\bf Case 2:}  $y_1\leq x_1$.

Since $y_1\leq x_1$, we can perform an $(a_1,b_1)^{y_1}$-derivation from 
$\{A,B\}$ to obtain (by Lemma~\ref{lem:1}) the $k$-irreducible pair $\{C,D\}$, where
\[C=\{(x_1-y_1)\cdot a_1,y_1\cdot (a_1-b_1),x_2\cdot a_2,\ldots,x_n\cdot a_n\}\]
and $D=\{y_2\cdot b_2,\ldots,y_m\cdot b_m\}$.

Since $r'=\max(C)+\max(D) \leq a_1+b_2<r$, it follows from the induction 
hypothesis that 
\begin{equation}\label{eq:c2}
|C|=(x_1-y_1)+y_1+\sum_{i=2}^n x_i=\sum_{i=1}^n x_i\leq \max(D)= b_2 .
\end{equation}
It follows from~\eqref{eq:c2} that $|A|=\sum_{i=1}^n x_i=|C|\leq b_2<b_1$. This 
concludes the first part of the proof.

\bs\n {\bf Part II:} In this part, we show that $|B|\leq \max(A)=a_1$.
Assume that $|B|>a_1$. Then since $a_1>b_1$ and $|A|\leq  b_1$ (by Part I), we obtain $|B|>|A|$.  
We now consider the cases $a_n>b_1$ and $b_1>a_n$. (Recall that 
$b_1\not=a_n$ since $A\cap B=\emptyset$.)

\bs\n{\bf Case 1:} $a_n>b_1$.

Then it follows from our general assumption~\eqref{eq:gen-ass} that 
\[a_1>\ldots>a_n>b_1>\ldots>b_m.\]
%Recall that   $|B|>|A|$ (from our assumption that $|B|>a_1$ above). 
We consider the following two subcases.

\bs\n{\bf Case 1.1:} $y_1>\sum_{i=1}^n x_i$. 

Then, we can perform an $\prod_{i=1}^n(a_i,b_1)^{x_i}$-derivation
from $\{A,B\}$ to obtain (by Lemma~\ref{lem:1})  the $k$-irreducible pair $\{C,D\}$, where
\[C=\{x_1\cdot (a_1-b_1),x_2\cdot (a_2-b_1),\ldots,x_n\cdot (a_n-b_1)\},\]
and 
\[D=\Big\{\Big(y_1-\sum_{i=1}^n x_i\Big)\cdot b_1,y_2\cdot b_2,\ldots,y_m\cdot b_m\Big\}.\]
Since $r'=\max(C)+\max(D)=(a_1-b_1)+b_1 <r$, it follows from the induction hypothesis that 
\begin{equation}\label{eq:c1.1}
|C|=\sum_{i=1}^n x_i\leq \max(D)\mbox{ and }|D|=\sum_{j=1}^m y_j-\sum_{i=1}^n x_i\leq \max(C).
\end{equation}
Thus, it follows from~\eqref{eq:c1.1}
\begin{equation*}
|B|=\sum_{j=1}^m y_j=|C|+|D|\leq \max(C)+\max(D)= (a_1-b_1)+b_1=a_1.
\end{equation*}

\bs\n{\bf Case 1.2:} $y_1\leq \sum_{i=1}^n x_i$.

Recall from the first paragraph in Part II that 
\[\sum_{j=1}^m y_j=|B| > |A|=\sum_{i=1}^n x_i.\] 
Consequently, the above inequality together with  $y_1\leq \sum_{i=1}^n x_i$ imply 
that  there exists an integer $t$, $1\leq t<m$, such that
\begin{equation}\label{eq:t}
\sum_{j=1}^t y_j\leq \sum_{i=1}^n x_i \mbox{ and } \sum_{j=1}^{t+1} y_j>\sum_{i=1}^n x_i.
\end{equation}
Then it follows from~Lemma~\ref{lem:2} that there exist integers $z_{ij}\geq0 $, $1\leq i\leq n$ 
and $1\leq j\leq t+1$, such that 
\[\sum_{i=1}^n z_{ij}=y_j\mbox{ for $1\leq j\leq t$, }\;
z_{i,t+1}=x_i-\sum_{j=1}^t z_{ij}\geq0, \mbox{ and } y_{t+1}>\sum_{i=1}^n z_{i,t+1}.\]
Thus, we can perform a $\prod_{i=1}^{n}\prod_{j=1}^{t+1}(a_i,b_j)^{z_{ij}}$-derivation from $\{A,B\}$ 
to obtain (by Lemma~\ref{lem:1})  the $k$-irreducible pair $\{C,D\}$, where
\begin{multline*}
C=\{z_{11}\cdot (a_1-b_1),\ldots,z_{1,t+1}\cdot (a_1-b_{t+1}),\ldots,\\
z_{i1}\cdot (a_i-b_1),\ldots,z_{i,t+1}\cdot (a_i-b_{t+1}),\ldots,\\
z_{n1}\cdot (a_n-b_1),\ldots,z_{n,t+1}\cdot (a_n-b_{t+1})\},
\end{multline*}
and
\[D=\Big\{\big(y_{t+1}-\sum_{i=1}^n z_{i,t+1}\big)\cdot b_{t+1},y_{t+2}\cdot b_{t+2},\ldots,y_{m}\cdot b_m\Big\}.\]
Since $a_1>\ldots>a_n>b_1>\ldots>b_m$, it follows that 
\[\max(C)\leq \max(A)-\min\limits_{1\leq j\leq t+1}b_j=a_1-b_{t+1} \mbox{ and } \max(D)=b_{t+1}.\]
Thus, $r'=\max(C)+\max(D) \leq (a_1-b_{t+1})+b_{t+1} <r$ and it follows from the induction
hypothesis that 
\begin{equation}\label{eq1:c1.2}
|C|=\sum_{j=1}^t \sum_{i=1}^n z_{ij}+\sum_{i=1}^n z_{i,t+1}=\sum_{j=1}^{t}y_j+\sum_{i=1}^n z_{i,t+1}\leq \max(D),
\end{equation}
and 
\begin{equation}\label{eq2:c1.2}
|D|=(y_{t+1}-\sum_{i=1}^n z_{i,t+1})+\sum_{j=t+2}^m y_j\leq \max(C).
\end{equation}
From~\eqref{eq1:c1.2} and~\eqref{eq2:c1.2}, we obtain
\begin{eqnarray*}\label{eq:7}
|B|=\sum_{j=1}^m y_j=|C|+|D|\leq \max(C)+\max(D)\leq a_1-b_{t+1}+b_{t+1}=a_1,
\end{eqnarray*}
as required. 

\bs\n{\bf Case 2:} $b_1>a_n$. 

 Let $s$ be that smallest index such that $b_1>a_s$. Since $a_1>b_1>a_n$, the integer $s$ exists 
 and $2\leq s\leq n$.  We consider the following two subcases.

\bs\n{\bf Case 2.1:} $y_1\leq \sum_{i=s}^n x_n$. 

Since $y_1\leq \sum_{i=s}^n x_n$, there exist integers $z_i\geq 0$, $s\leq i\leq n$ such that 
$x_i\geq z_i$, and $y_1=\sum_{i=s}^nz_i$. We can perform 
an $\prod_{i=s}^n (a_i,b_1)^{z_i}$-derivation from $\{A,B\}$ to obtain (by Lemma~\ref{lem:1}) 
the $k$-irreducible pair $\{C,D\}$, where
\[C=\{x_1\cdot a_1,\ldots,x_{s-1}\cdot a_{s-1},(x_s-z_s)\cdot a_s,\ldots,(x_n-z_n)\cdot a_n\},\]
and 
\[ D=\{z_s\cdot (b_1-a_s),\ldots,z_n\cdot (b_1-a_n),y_2\cdot b_2,\ldots,y_m\cdot b_m\}.\]
Since $r'=\max(C)+\max(D)\leq a_1+\max\{b_1-a_n,b_2\}<r$, it follows from the induction 
hypothesis that 
\begin{equation}\label{eq:2.1}
|D|=\sum_{i=s}^nz_i+\sum_{j=2}^m y_j=y_1+\sum_{j=2}^m y_j=\sum_{j=1}^m y_j\leq \max(C)=a_1.
\end{equation}
Thus, it follows from~\eqref{eq:2.1} that $|B|=\sum_{j=1}^m y_i=|D|\leq a_1$ as required. 

\bs\n{\bf Case 2.2:} $y_1>\sum_{i=s}^n x_n$. 

Since $y_1>\sum_{i=s}^n x_n$,  we can perform an $\prod_{i=s}^n (a_i,b_1)^{x_i}$-derivation 
from $\{A,B\}$ to obtain (by Lemma~\ref{lem:1}) the $k$-irreducible pair $\{A',B'\}$, where
\[A'=\{x_1\cdot a_1,x_2\cdot a_2,\ldots,x_{s-1}\cdot a_{s-1}\},\]
and 
\[B'=\Big\{(y_1-\sum_{i=s}^n x_n)\cdot b_1,x_s\cdot (b_1-a_s),\ldots,x_n\cdot (b_1-a_n),
y_2\cdot b_2,\ldots,y_m\cdot b_m\Big\}.\]
Note that $\max(B')=b_1$. We can now rename the distinct elements of the multiset $B'$ as 
$b_1',\ldots,b'_{m'}$ such that $\max(B')=b_1'>\ldots>b'_{m'}=\min(B')$. Let $y'_j$ be
the multiplicity of $b'_j$ for $1\leq j\leq m'$. We also let $a'_i=a_i$ for $1\leq i\leq s-1=n'$.

Recall from Part I  that $|A|\leq \max(B)=b_1$. Hence, 
\[|A'|=\sum_{i=1}^{s-1} x_i\leq \sum_{i=1}^n x_i=|A|\leq b_1.\]
If $|B'|\leq |A'|$, then $|B|=\sum_{j=1}^m y_j=|B'|\leq |A'|\leq b_1<a_1$, and we are done. 
So, we may assume that $|B'|>|A'|$. Since $a'_{n'}=a_{s-1}>b_1= b'_1$ (owing to the definition 
of $s$ and the fact that $A\cap B=\emptyset$),  it follows that 
\[ a'_1>\ldots>a'_{n'}>b'_1>\ldots>b'_{m'}.\] 
We can now proceed as in Part II (Case~1 ) to infer that 
\[|B'|\leq \max(A')\Longrightarrow |B|=\sum_{j=1}^m y_j=|B'|\leq \max(A')=a_1.\]
This concludes the second part of the proof.

\bs
We conclude from Part I and Part II  that
\[|A|\leq \max(B)=b_1\quad \mbox{ and }\quad  |B|\leq \max(A)=a_1.\]
Moreover, these inequalities imply that
\[\ell(A,B)=|A|+|B|\leq b_1+a_1\leq 2k-1,\]
where the last inequality follows from the fact that $1\leq b_1<a_1\leq k$. Finally, 
since $\ell(k)\geq 2k-1$ (see the example in~\eqref{exp:lb} from Section~\ref{sec:intro}), 
it follows that $\ell(k)=2k-1$.
\end{proof}
We now prove the corollary.
\begin{proof}[Proof of Corollary~\ref{cor:1}]
Let $A=\{x_1\cdot a_1,x_2\cdot a_2,\ldots, x_n\cdot a_n\}$ and 
$B=\{y_1\cdot b_1,y_2\cdot b_2,\ldots, y_m\cdot b_m\}$ be multisets, where 
the $a_i$'s and $b_i$'s are all positive integers 
such that $1\leq a_i,b_j\leq k$ for $1\leq i\leq n$, $1\leq j\leq m$.
Moreover, $x_i>0$ and $y_j>0$ are the multiplicities of $a_i$ and $b_j$ respectively.
We also assume that the $a_i$'s (resp. $b_j$'s) are pairwise distinct.
Without loss of generality, we may also assume that $A\cap B=\emptyset$ and $a_1>b_1$.

Suppose that $\{A,B\}$ is a $k-$irreducible pair such that $\ell(A,B)=2k-1$. 
Then it follows from Theorem~\ref{thm:1} (and the above setup) that
\begin{equation}\label{eq:cor0}
|A|=\max(B)=b_1=k-1 \mbox{ and } |B|=\max(A)=a_1=k. 
\end{equation}
For a proof by contradiction assume that the pair $\{A,B\}$ is different 
from the pair $\left\{\{k\cdot(k-1)\},\{(k-1)\cdot k\}\right\}$. 
We consider two cases.

\bs
\n{\bf Case 1:} $x_1\geq y_1$.

We perform an $(a_1,b_1)^{y_1}$-derivation from 
$\{A,B\}$ to obtain (by Lemma \ref{lem:1}) the $k$-irreducible pair $\{C,D\}$, where
\[C=\{(x_1-y_1)\cdot a_1,y_1\cdot (a_1-b_1),x_2\cdot a_2,\ldots,x_n\cdot a_n\}\]
and $D=\{y_2\cdot b_2,\ldots,y_m\cdot b_m\}$.

Since $a_1>b_1$, $y_1>0$, and $\sum A=\sum B$, we have $m>1$, so that 
$b_2\in D$. Hence, $C$ and $D$ are both nonempty.
We now use Theorem~\ref{thm:1} on the irreducible pair $\{C,D\}$ to infer that
\begin{equation}\label{eq:cor1}
|C|=(x_1-y_1)+y_1+\sum_{i=2}^n x_i=\sum_{i=1}^n x_i\leq \max(D)= b_2.
\end{equation}
It follows from~\eqref{eq:cor1} that $|A|=\sum_{i=1}^n x_i=|C|\leq b_2<b_1=k-1$. 
This contradicts the fact that $|A|=b_1=k-1$ (see~\eqref{eq:cor0}).

\bs
\n{\bf Case 2:} $y_1>x_1$.

We perform an $(a_1,b_1)^{x_1}$-derivation from $\{A,B\}$ to 
obtain (by Lemma \ref{lem:1}) the $k$-irreducible pair $\{C,D\}$, where
\[C=\{x_1\cdot (a_1-b_1),x_2\cdot a_2,\ldots,x_n\cdot a_n\}
=\{x_1\cdot 1,x_2\cdot a_2,\ldots,x_n\cdot a_n\},\]
and $D=\{ (y_1-x_1)\cdot b_1,y_2\cdot b_2,\ldots,y_m\cdot b_m\}$.

If $n=1$, then $x_1=|A|=k-1$. So $y_1>x_1$ and $\ell(A,B)=2k-1$ imply $y_1=k$,
contradicting that  $\{A,B\}$ is different from $\left\{\{k\cdot(k-1)\},\{(k-1)\cdot k\}\right\}$. Thus we may assume that $n\geq 2$, that is, $a_2\in C$.

Since $a_2\not=b_1=k-1$, we must have $z=b_1-a_2>0$. If $z<x_1$ also holds, then $C'=\{a_2,z\cdot 1\}\subset C$ and $D'=\{b_1\}\subset D$ form a witness for the reducibility of $\{C,D\}$, which is a contradiction. Thus, we must have 
$b_1-a_2\geq x_1$.
We now use Theorem~\ref{thm:1} on the irreducible pair $\{C,D\}$ to infer that
\begin{equation}\label{eq:cor2}
|D|=(y_1-x_1)+\sum_{j=2}^m y_j=-x_1+\sum_{j=1}^m y_j\leq \max(C)=a_2\leq b_1-x_1.
\end{equation}
It follows from~\eqref{eq:cor2} that $|B|=\sum_{j=1}^m y_j=x_1+|D|\leq b_1=k-1$.
This contradicts the fact that $|B|=a_1=k$ (see~\eqref{eq:cor0}).
\end{proof}
%================================================================================
\section{Concluding Remarks}\label{sec:conc}
%========================================================================================

One may wonder if our results can be extended to other infinite abelian groups.
For instance, consider irreducible pairs $\{A,B\}$, where $A$ and $B$ are multisets 
of rational numbers. Are there suitable 
(and general enough) conditions on the elements of $\{A,B\}$ that will guarantee 
that $\ell(A,B)$ is finite?

Finally, we remark that Theorem~\ref{thm:1} can be used to
bound the number of {\em $\lambda$-fold vector space partitions} (e.g., see~\cite{ESSSV}).
We shall address this application in a subsequent paper.

\bs\n{\bf Acknowledgement:}\
The authors thank G. Seelinger, L. Spence, and C. Vanden Eynden for providing useful
 suggestions that led to an improved version of this paper.
%======================================================================================

%

\bs\n \hrule
%================================================================================
\end{document}